\documentclass[12pt,a4paper]{amsart}

\usepackage[a4paper,hmargin=32mm]{geometry}

\usepackage{mathtools}
\usepackage{libertinus-otf}
\usepackage{microtype}
\usepackage{booktabs}

\usepackage{algorithm}
\usepackage{algpseudocode}

\usepackage[hidelinks]{hyperref}

\numberwithin{equation}{section}

\newtheorem{theorem}{Theorem}[section]
\newtheorem{lemma}[theorem]{Lemma}
\newtheorem{proposition}[theorem]{Proposition}
\newtheorem{corollary}[theorem]{Corollary}

\newcounter{lettertheorem}

\newtheorem{theoremletter}[lettertheorem]{Theorem}

\newtheorem{corollaryletter}[lettertheorem]{Corollary}

\theoremstyle{definition}

\newtheorem{example}[theorem]{Example}

\theoremstyle{remark}
\newtheorem{remark}[theorem]{Remark}

\usepackage{xcolor}
\usepackage{tikz}
\usepackage{tikz-cd}
\usetikzlibrary{calc}

\usepackage{svg}

\colorlet{scolor}{red}
\colorlet{tcolor}{blue}
\definecolor{ucolor}{RGB}{100,200,80}

\newcommand{\reds}{\ensuremath{\mathord{{\color{scolor}s}}}}
\newcommand{\bluet}{\ensuremath{\mathord{{\color{tcolor}t}}}}
\newcommand{\greenu}{\ensuremath{\mathord{{\color{ucolor}u}}}}

\newcommand{\alphas}{\ensuremath{{\color{scolor}\alpha_s}}}
\newcommand{\alphat}{\ensuremath{{\color{tcolor}\alpha_t}}}
\newcommand{\alphau}{\ensuremath{{\color{ucolor}\alpha_u}}}
\newcommand{\alphax}{\ensuremath{{\color{genericcolor}\alpha_x}}}

\tikzset{
  soergel/.style={
    line cap=round,
    line join=round
  },
  boundary/.style={
    draw=black,
    dashed,
    line width=0.35pt
  },
  strand/.style={
    line width=0.8pt,
    line cap=round,
    line join=round
  },
  sline/.style={
    strand,
    draw=scolor
  },
  tline/.style={
    strand,
    draw=tcolor
  },
  uline/.style={
    strand,
    draw=ucolor
  },
  sdot/.style={
    circle,
    fill=scolor,
    draw=scolor,
    inner sep=0pt,
    minimum size=1.1mm
  },
  tdot/.style={
    circle,
    fill=tcolor,
    draw=tcolor,
    inner sep=0pt,
    minimum size=1.1mm
  },
  udot/.style={
    circle,
    fill=ucolor,
    draw=ucolor,
    inner sep=0pt,
    minimum size=1.1mm
  },
  morphism box/.style={
    rectangle,
    draw=black,
    fill=white,
    line width=0.4pt,
    inner xsep=5pt,
    inner ysep=3pt
  }
}

\definecolor{genericcolor}{rgb}{0.55,0.20,0.65}

\newcommand{\anyx}{%
  \ensuremath{\mathord{{\color{genericcolor}x}}}%
}

\tikzset{
  genericline/.style={
    strand,
    draw=genericcolor
  },
  genericdot/.style={
    circle,
    fill=genericcolor,
    draw=genericcolor,
    inner sep=0pt,
    minimum size=1.1mm
  }
}

\newcommand{\dotStart}{%
\begin{tikzpicture}[
  soergel,
  x=.65cm,
  y=.65cm,
  baseline=-.6ex
]
  \draw[boundary] (-1,-1)--(1,-1);
  \draw[boundary] (-1, 1)--(1, 1);

  \draw[genericline] (0,-1)--(0,0);
  \node[genericdot] at (0,0) {};
\end{tikzpicture}%
}

\newcommand{\dotFinish}{%
\begin{tikzpicture}[
  soergel,
  x=.65cm,
  y=.65cm,
  baseline=-.6ex
]
  \draw[boundary] (-1,-1)--(1,-1);
  \draw[boundary] (-1, 1)--(1, 1);

  \draw[genericline] (0,0)--(0,1);
  \node[genericdot] at (0,0) {};
\end{tikzpicture}%
}

\newcommand{\mergeVertex}{%
\begin{tikzpicture}[
  soergel,
  x=.65cm,
  y=.65cm,
  baseline=-.6ex
]
  \draw[boundary] (-1,-1)--(1,-1);
  \draw[boundary] (-1, 1)--(1, 1);

  \draw[genericline] (-.55,-1)--(0,0);
  \draw[genericline] ( .55,-1)--(0,0);
  \draw[genericline] (0,0)--(0,1);
\end{tikzpicture}%
}

\newcommand{\splitVertex}{%
\begin{tikzpicture}[
  soergel,
  x=.65cm,
  y=.65cm,
  baseline=-.6ex
]
  \draw[boundary] (-1,-1)--(1,-1);
  \draw[boundary] (-1, 1)--(1, 1);

  \draw[genericline] (0,-1)--(0,0);
  \draw[genericline] (0,0)--(-.55,1);
  \draw[genericline] (0,0)--( .55,1);
\end{tikzpicture}%
}

\newcommand{\vertexSU}{%
\begin{tikzpicture}[
  soergel,
  x=.70cm,
  y=.70cm,
  baseline=-.6ex
]
  \draw[boundary] (-1.2,-1)--(1.2,-1);
  \draw[boundary] (-1.2, 1)--(1.2, 1);

  \draw[sline] (-.45,-1)--(0,0);
  \draw[uline] ( .45,-1)--(0,0);

  \draw[uline] (0,0)--(-.45,1);
  \draw[sline] (0,0)--( .45,1);
\end{tikzpicture}%
}

\newcommand{\vertexTU}{%
\begin{tikzpicture}[
  soergel,
  x=.70cm,
  y=.70cm,
  baseline=-.6ex
]
  \draw[boundary] (-1.5,-1)--(1.5,-1);
  \draw[boundary] (-1.5, 1)--(1.5, 1);

  \draw[tline] (-.8,-1)--(0,0);
  \draw[uline] ( 0,-1)--(0,0);
  \draw[tline] ( .8,-1)--(0,0);

  \draw[uline] (0,0)--(-.8,1);
  \draw[tline] (0,0)--(0,1);
  \draw[uline] (0,0)--(.8,1);
\end{tikzpicture}%
}

\newcommand{\vertexST}{%
\begin{tikzpicture}[
  soergel,
  x=.70cm,
  y=.70cm,
  baseline=-.6ex
]
  \draw[boundary] (-2.25,-1)--(2.25,-1);
  \draw[boundary] (-2.25, 1)--(2.25, 1);

  \draw[sline] (-1.6,-1)--(0,0);
  \draw[tline] (-.8,-1)--(0,0);
  \draw[sline] (0,-1)--(0,0);
  \draw[tline] (.8,-1)--(0,0);
  \draw[sline] (1.6,-1)--(0,0);

  \draw[tline] (0,0)--(-1.6,1);
  \draw[sline] (0,0)--(-.8,1);
  \draw[tline] (0,0)--(0,1);
  \draw[sline] (0,0)--(.8,1);
  \draw[tline] (0,0)--(1.6,1);
\end{tikzpicture}%
}
\newcommand{\polynomialBox}[1]{%
\begin{tikzpicture}[
  soergel,
  x=.65cm,
  y=.65cm,
  baseline=-.6ex
]
  \draw[boundary] (-1,-1)--(1,-1);
  \draw[boundary] (-1, 1)--(1, 1);

  \node[morphism box] at (0,0) {$#1$};
\end{tikzpicture}%
}

\input{double-leaf-morphism.tex}

\newcommand{\exampleDecoratedDoubleLeafPicture}{%
\begin{tikzpicture}[baseline=(pic.center)]
  \node[inner sep=0pt, outer sep=0pt] (pic)
    {\exampleDoubleLeafPicture};
  \node[
    morphism box,
    anchor=east,
    inner xsep=3pt,
    inner ysep=2pt
  ]
  at ($(pic.center)+(-0.85cm,-0.45cm)$)
  {$
    \begin{gathered}
      (-20-29\varphi)\alphas+
      \\(-10-17\varphi)\alphat +
      \\(1+2\varphi)\alphau
    \end{gathered}
  $};
\end{tikzpicture}%
}

\input{rex-morphisms.tex}

\newcommand{\Hec}{\mathcal{H}}

\newcommand{\Hom}{\operatorname{Hom}}

\newcommand{\expr}[1]{\mathbf{#1}}

\newcommand{\BSB}{\mathrm{BSBim}_{H_3}}

\newcommand{\QG}{\left(\Omega_{Q} H_{3}\right)_{\oplus}}

\newcommand{\CF}{\mathcal{F}}

\newcommand{\LL}{\mathbb{LL}}

\newcommand{\dft}{\operatorname{defect}}

\newcommand{\bk}{\mathbb{k}}

\newcommand{\ch}{\mathcal{h}}

\title{Explicit \(H_3\) relation in the diagrammatic Hecke category}

\author{Lev Nechitailo}
\address{Department of Mathematics, HSE, Moscow, Russia}
\email{nechitailo.lev@gmail.com}

\date{27.09.2026}

\begin{document}

\begin{abstract}
We compute the missing three-color relation in type \(H_{3}\) as an explicit \(R\)-linear combination of double leaves. The expression has \(43424\) nonzero polynomial coefficients with scalars in \(\mathbb{Z}[\varphi]\), where \(\varphi = (1 + \sqrt{5})/2\). We verify the relation using exact arithmetic in the computer algebra system GAP. This completes the diagrammatic presentation of the Hecke category, defined in \cite{EW16}.
\end{abstract}

\maketitle

\section{Introduction}

Soergel bimodules form a categorification of the Hecke algebra and provide an incarnation of the Hecke category. This category gives an algebraic framework for Kazhdan--Lusztig theory and has become a fundamental tool in geometric and modular representation theory. In \cite{EW16}, Elias and Williamson gave a diagrammatic presentation of the Hecke category by generators and relations. This presentation makes it possible to work with morphisms between Soergel bimodules entirely diagrammatically.

There is, however, one exceptional case left unresolved in this presentation. The three-color, or Zamolodchikov, relations are attached to finite rank-three parabolic subgroups. For parabolics of types \(I_{2}(m)\times A_{1},A_{3},B_{3}\), one can choose two suitable rex paths for which the corresponding rex morphisms are equal. In the remaining type \(H_{3}\), this is no longer possible. See \cite[\S 1.4.3]{EW16}. 

At the level of Soergel bimodules, the Soergel Hom formula implies that the difference must be an \(R\)-linear combination of double leaves factoring through lower terms, but the explicit combination was unknown. Consequently, the \(H_{3}\) relation was missing from the diagrammatic presentation. The purpose of this paper is to compute this relation.

Let \(\varphi = \frac{1 + \sqrt{5}}{2}\). Consider the standard geometric realization of the Coxeter system \(H_{3}\) over \(\mathbb{Q}(\varphi)\). Let \(w_{0}\) be the longest element, and fix the following two reduced expressions of it
\[\expr{x} = ut st us ts tu st st s, \qquad \expr{y} = ststsutstsutstu.\]
Put \(R = \mathbb{Q}(\varphi)[\alpha_{s},\alpha_{t},\alpha_{u}],\) and write \(\widetilde\Lambda\) for the localization functor on the diagrammatic category with only the one-color and two-color relations imposed, before the \(H_{3}\) relation is imposed. Let \[P_{L},P_{R}\colon \expr{x} \longrightarrow \expr{y}\]
be the two rex morphisms appearing in the \(H_{3}\) relation of \cite[\S 1.4.3]{EW16}. The coefficient certificate accompanying this paper defines an explicit \(R\)-linear combination \(Z_{H_{3}}\) of double leaves from \(\expr{x}\) to \(\expr{y}\).
\begin{theoremletter}\label{th:A}
  For the standard geometric realization of \(H_{3}\),
  \[\widetilde\Lambda(P_{L}) - \widetilde\Lambda(P_{R}) = \widetilde\Lambda(Z_{H_{3}}).\]
\end{theoremletter}

This is the main theorem of the paper. The relation is too large to print: it has \(43424\) nonzero polynomial coefficients and is stored as \(54452\) monomial records. All scalar coefficients lie in \(\mathbb{Z}[\varphi]\). We will refer to this data as the \textit{certificate}. The complete file as well as the source code of producer and the GAP verifier can be found in the GitHub repository \cite{H3Code}. 

Using the geometric realization for the computations does not, actually, restrict the applicability. For another realization of \(H_{3}\) over any commutative domain, a diagonal rescaling of the simple roots and coroots gives the Cartan matrix in the balanced form with a parameter \(\rho\) satisfying \[\rho^{2} - \rho - 1 = 0.\] In this normalization the relation is obtained from the certificate by the base change
\[\mathbb{Z}[\varphi] \longrightarrow \bk, \qquad \varphi \longmapsto \rho,\]
together with the corresponding rescaling of the simple roots.
As a consequence, the \(H_{3}\) relation can be inserted into the general diagrammatic presentation. Namely, let \((W,S)\) be a finite-rank Coxeter system and let \(\ch\) be a balanced realization over a commutative domain \(\bk\). Put
  \[R = \operatorname{Sym}(\ch^{*}), \qquad Q = \operatorname{Frac}(R).\]
  Assume that, for every finite dihedral parabolic subgroup, the relevant two-color Jones-Wenzl projectors exist and are rotatable, as required in \cite{EW23}.

\begin{corollaryletter}
  Let \(\Hec_{BS}(W)\) be the Bott--Samelson diagrammatic category defined by the one-color, two-color, and three-color relations of \cite{EW16}, with the relation of Theorem~\ref{th:A}, transported by the normalization of Subsection~\ref{ss:norm}, imposed for every \(H_{3}\)-parabolic subgroup. Let \(\Hec(W)\) be its additive graded Karoubi envelope. Then:
  \begin{enumerate}
  \item there is a well-defined monoidal functor
    \[\Lambda_{W}\colon \Hec_{BS}(W) \longrightarrow \left(\Omega_{Q}W\right)_{\oplus},\]
    and the double leaves are linearly independent over \(R\);
  \item if \(\ch\) is Demazure surjective, the double leaves form an \(R\)-basis of every morphism space in \(\Hec_{BS}(W)\).
  \item if, in addition, \(\bk\) is a complete local domain, the diagrammatic character induces an isomorphism
    \[K_{0}^{\oplus}(\Hec(W)) \cong \mathbf{H}_{W}\]
    of \(\mathbb{Z}[v,v^{-1}]\)-algebras.
  \end{enumerate}
\end{corollaryletter}

The corresponding statements for unbalanced realizations are obtained by making the modifications described in \cite[\S 7]{EW23}.

The computation uses several reductions. First, after localization the morphism spaces split into blocks indexed by endpoints in \(H_{3}\). Then we use the triangularity of double leaves in path-dominance order. This allows us to carry out the computation by induction on endpoint length. Finally, instead of performing symbolic arithmetic with rational functions, we evaluate the relevant matrices at exact integer points and recover the coefficients by interpolation. All of this is the content of Section~\ref{sec:comp}.

The implementation and verification are described in Section~\ref{sec:implementation}. The producer is written in Rust, and the independent GAP verifier checks the resulting certificate. We also provide substantially faster probabilistic verification based on Freivalds' algorithm.

In Section~\ref{sec:con} we explain the normalization for arbitrary \(H_{3}\) realizations and derive the categorification consequences stated above. 

The appendix describes the format of the certificate, including the rex paths, interpolation points, and polynomial coefficients. The certificate together with the source code is provided in the github repository.

\subsection*{Tool and computational resource disclosure}

Large language models, namely OpenAI's ChatGPT and Codex, were used during the preparation of this work. Their use included language editing, discussion and criticism of intermediate arguments, and substantial assistance with the development, debugging, review, and optimization of the computational code.

The mathematical outline and overall computational strategy were developed by the author and were provided to an AI agent for implementation in Rust. However, the source code contained in the repository is far from being the first output of the model. It was obtained through a long process of correcting and directing the AI agent. It also required extensive profiling and optimization loops to achieve a feasible computation time.

The author did not manually inspect and verify every part of the Rust producer. The correctness of the main result, however, does not rely on the correctness of this implementation. The resulting certificate is independently checked by GAP. The author has inspected and verified the GAP code and takes full responsibility for it as well as for all mathematical statements and arguments of the paper.

\subsection*{Acknowledgments}

The author would like to thank his friends Igor Shimanogov for suggesting the optimization workflow, and Maksim Terekhov for kindly providing his powerful PC for tests and computations. The author also thanks his adviser Anton Khoroshkin for introducing him to a wonderful book \cite{EMTW20}, which inspired him to study diagrammatic calculus, as well as for everything he has taught him and for his help in many different ways.

\section{Setup}
\subsection{The geometric realization}\label{ss:real}
Let \(H_{3} = (W, S)\) be the Coxeter system with \[S = \{\reds, \bluet, \greenu\}, \qquad m_{\reds\bluet} = 5, \quad m_{\bluet\greenu} = 3, \quad m_{\reds\greenu} =2.\]
We use the standard geometric realization over \(K = \mathbb{Q}(\varphi)\), where \(\varphi = \frac{1 +\sqrt{5}}{2}\). With respect to the ordered set \(\reds, \bluet, \greenu\), its Cartan matrix is 
\[A = (\alpha_{i}^{\vee}, \alpha_{j}) = \begin{pmatrix}
        2 & - \varphi & 0 \\
        - \varphi & 2 & -1 \\
        0 & -1 & 2 \end{pmatrix}.\]
    Let \(V\) denote the corresponding reflection representation, and set \(R = \operatorname{Sym}(V^{*})\) and \(Q = \operatorname{Frac}(R)\). We give every simple root degree \(2\). Let
    \[\Phi = \{w(\alpha_{r}) \mid w \in H_{3}, \ r \in S\}.\]We also put
    \[A_{0} = \mathbb{Z}[\varphi].\]
    The calculation is performed over \(K\), but every coefficient in the final relation belongs to \(A_{0}[\alpha_{s},\alpha_{t},\alpha_{u}].\)
    The standard geometric realization is faithful and balanced, it also satisfies Demazure surjectivity. We discuss the dependence on the realization later in \ref{sec:con}.

\subsection{The pre-Hecke category and its localization}
    Let \(\widetilde{\Hec}_{BS}(H_{3})\) be the \(K\)-linear strict monoidal category generated by the objects \(S = \{\reds, \bluet, \greenu\}\) and the following morphisms. The violet color denotes an arbitrary color \(\anyx\in S\). These are the usual generators of the diagrammatic Hecke category, see \cite{EW16}. 

\par\medskip
\noindent
\textbf{Univalent vertices (degree \(1\)).}

\[
\begin{array}{
  c@{\qquad}l
  @{\qquad\qquad}
  c@{\qquad}l
}
\begin{array}{c}
  \dotStart \\[3ex]
  \text{``startdot''}
\end{array}
&
:\;(\anyx)\longmapsto\emptyset,
&
\begin{array}{c}
  \dotFinish \\[3ex]
  \text{``enddot''}
\end{array}
&
:\;\emptyset\longmapsto(\anyx).
\end{array}
\]

\par\medskip
\noindent
\textbf{Trivalent vertices (degree \(-1\)).}

\[
\begin{array}{
  c@{\qquad}l
  @{\qquad\qquad}
  c@{\qquad}l
}
\begin{array}{c}
  \mergeVertex \\[3ex]
  \text{``merge''}
\end{array}
&
:\;(\anyx,\anyx)\longmapsto(\anyx),
&
\begin{array}{c}
  \splitVertex \\[3ex]
  \text{``split''}
\end{array}
&
:\;(\anyx)\longmapsto(\anyx,\anyx).
\end{array}
\]

\par\medskip
\noindent
\textbf{Two-colored vertices (degree \(0\)).}

\[
\begin{array}{c@{\qquad}l}
\begin{array}{cc}
  G_{\reds,\greenu}: = &  \vertexSU \\[3ex]
 & \text{``4-valent vertex''}
\end{array}
&
:\;
(\reds,\greenu)
\longmapsto
(\greenu,\reds),
\qquad
m_{\reds \greenu}=2,
\\[1.5em]
\begin{array}{cc}
   G_{\bluet,\greenu}: = & \vertexTU \\[3ex]
  &\text{``6-valent vertex''}
\end{array}
&
:\;
(\bluet,\greenu,\bluet)
\longmapsto
(\greenu,\bluet,\greenu),
\qquad
m_{\bluet \greenu}=3,
\\[1.5em]
\begin{array}{cc}
   G_{\reds,\bluet}: = & \vertexST \\[3ex]
  &\text{``10-valent vertex''}
\end{array}
&
:\;
(\reds,\bluet,\reds,\bluet,\reds)
\longmapsto
(\bluet,\reds,\bluet,\reds,\bluet),
\qquad
m_{\reds \bluet}=5.
\end{array}
\]

\par\medskip
\noindent
\textbf{Polynomial boxes (degree \(\deg f\)).}
For every homogeneous polynomial \(f\in R\),

\[
\begin{array}{c@{\qquad}l}
  \polynomialBox{f}
&
:\;\emptyset\longrightarrow\emptyset,
\qquad f\in R.
\end{array}
\]

In \(\widetilde{\Hec}_{BS}(H_{3})\) we impose only one-color and two-color relations, whose complete list can be found in \cite{EW16}. Note that, for now, we do not impose the three-color relation.

To describe the localization, let us recall the \(Q\)-groupoid category of \cite{EW23}. The category \(\Omega_{Q}H_{3}\) has one object \(r_{w}\) for every \(w \in H_{3}\), with
\[\Hom_{\Omega_{Q}H_{3}}(r_{x},r_{y}) =
  \begin{cases}
    Q, \quad x = y, \\
    0, \quad x \neq y.
  \end{cases}
\]
Its monoidal structure is induced by multiplication in \(H_{3}\) and the natural action of \(H_{3}\) on \(Q\). Thus
\[r_{x} \otimes r_{y} = r_{xy}.\]
We work with its additive closure \(\QG\).

Even without the three-color relation, the one-color and two-color assignments define a monoidal functor
\[\widetilde\Lambda \colon \widetilde{\Hec}_{BS}(H_{3}) \to \QG, \qquad (x) \mapsto r_{\operatorname{id}} \oplus r_{x}.\]
Thus the image of the expression  \( \expr{w} = (s_{1}, s_{2}, \dots, s_{m})\) is given by the sum indexed by subexpressions \(\expr{e} \in \{0,1\}^{m}\) of the expression  \(\expr{w}\). Put
\[\expr{w}^{\expr{e}} := s_{1}^{e_{1}}\dots s_{m}^{e_{m}}.\]
Then 
 \[\widetilde \Lambda (\expr{w}) = \bigoplus_{\expr{e} \subset \expr{w}} r_{\expr{w}^\expr{e}}\]
Morphisms between such objects can therefore be represented by matrices. On the one-colored generating morphisms, the functor \(\widetilde{\Lambda}\) is given by

\begin{subequations}\label{subeq:Lambda-one-color}
\begin{gather}
\label{eq:Lambda-f}
\begin{array}{c}
\polynomialBox{f}
\end{array}
\mapsto
\begin{array}{c|c}
 & \emptyset\\
\hline
\emptyset & f
\end{array}
\\[1em]
\label{eq:Lambda-startdot}
\begin{array}{c}
\dotStart
\end{array}
\mapsto
\begin{array}{c|cc}
 & 0 & 1\\
\hline
\emptyset & \alphax & 0
\end{array}
\\[1em]
\label{eq:Lambda-enddot}
\begin{array}{c}
\dotFinish
\end{array}
\mapsto
\begin{array}{c|c}
 & \emptyset\\
\hline
0 & 1\\
1 & 0
\end{array}
\\[1em]
\label{eq:Lambda-split}
\begin{array}{c}
\splitVertex
\end{array}
\mapsto
\begin{array}{c|cc}
 & 0 & 1\\
\hline
00 & 1/\alphax & 0\\
01 & 0 & 1/\alphax\\
10 & 0 & -1/\alphax\\
11 & -1/\alphax & 0
\end{array}
\\[1em]
\label{eq:Lambda-merge}
\begin{array}{c}
\mergeVertex
\end{array}
\mapsto
\begin{array}{c|cccc}
 & 00 & 01 & 10 & 11\\
\hline
0 & 1 & 0 & 0 & 1\\
1 & 0 & 1 & 1 & 0
\end{array}
\end{gather}
\end{subequations}

In order to define the functor on two-colored generators, let 
\[
\expr{w}_{x,y}=(x,y,x,y,\dots)
\quad\text{and}\quad
\expr{w}_{y,x}=(y,x,y,x,\dots)
\]
be the alternating expressions of length \(m_{xy}\). The matrix of \(\widetilde{\Lambda}(G_{x,y})\) has columns indexed by
subexpressions \(\expr{e}\subset \expr{w}_{x,y}\) and rows
indexed by subexpressions
\(\expr{e}'\subset \expr{w}_{y,x}\).
Define
\[
\pi_{x,y}
:=
\prod_{i=1}^{m_{xy}}
x_{1}x_{2}\cdots x_{i-1}(\alpha_{x_{i}}),
\qquad
\text{where }(x_{1},\dots,x_{m_{xy}})=\expr{w}_{x,y},
\]
and
\[
\zeta_{x,y}(\expr{e}')
:=
\prod_{i=1}^{m_{xy}}
y_{1}^{e'_{1}}y_{2}^{e'_{2}}\cdots y_{i-1}^{e'_{i-1}}(\alpha_{y_{i}}),
\qquad
\text{where }(y_{1},\dots,y_{m_{xy}})=\expr{w}_{y,x}.
\]

Then the matrix coefficients of \(\widetilde{\Lambda}(G_{x,y})\) are
\begin{equation}\label{eq:Lambda-two-color-general}
\bigl(\widetilde{\Lambda}(G_{x,y})\bigr)_{\expr{e}}^{\expr{e}'}
=
\begin{cases}
\dfrac{\pi_{x,y}}{\zeta_{x,y}(\expr{e}')}
&
\text{if }
\expr{w}_{x,y}^{\,\expr{e}}
=
\expr{w}_{y,x}^{\,\expr{e}'},
\\[1em]
0
&
\text{otherwise.}
\end{cases}
\end{equation}
This is the formula of \cite[Lemma~2.11]{EW23}.
The one-color and two-color relations are satisfied by these assignments as proven in \cite{EW23}.

For the geometric realization fixed in Subsection~\ref{ss:real}, let
\[\widetilde{\CF}\colon \widetilde{\Hec}_{BS}(H_{3})\longrightarrow \mathrm{BSBim}_{H_3} \]
denote the standard functor sending the diagrammatic generators to the corresponding morphisms of Bott--Samelson bimodules. The action of \(H_{3}\) on \(Q\) is faithful, hence the constructions fit into the commutative diagram
\begin{equation}
  \label{eq:comm-sq}
  \begin{tikzcd}[column sep=large, row sep=large]
    \widetilde{\Hec}_{BS}(H_{3})
    \arrow[r, "\widetilde{\CF}"]
    \arrow[d, "\widetilde\Lambda"]
    &
    \BSB \arrow[d, "{(-)\otimes_R Q}"]
    \\
    \QG
    \arrow[r, "\sim"]
    &
    \mathrm{StdBim}_{Q,H_3}
  \end{tikzcd}
\end{equation}
For a discussion of the lower horizontal arrow, see \cite[\S 5]{EW23}. 
\subsection{The missing relation}
The longest element in \(H_{3}\) has length 15. A direct computation shows that its rex graph has \(286\) vertices and \(640\) edges. We fix two reduced expressions \[\expr{x} = utstuststuststs, \ \expr{y} = ststsutstsutstu\]
of \(w_{0}\), and two rex morphisms between them 
\[P_{L} =\rexLeftPretty , \qquad P_{R} = \rexRightPretty.\]
\begin{remark}
  These are exactly the two morphisms used in \cite[\S 1.4.3]{EW16}. However, our color convention differs. In their picture, \(s\) is green and \(u\) is red. 
\end{remark}
\begin{remark}
  The corresponding sequences of braid moves are fixed in the certificate, see Appendix~\ref{app:cert}.
\end{remark}
For each \(u \in H_{3}\), fix a reduced expression \(\expr{u}\). For every subexpression \(\expr{e} \subset \expr{x}\) with endpoint \(\expr{x}^{\expr{e}} = u\), fix the corresponding light leaf
\[LL_{\expr{e},\expr{x}} \colon \expr{x} \longrightarrow \expr{u}.\]
Likewise, for every \(\expr{f} \subset \expr{y}\) with endpoint \(\expr{y}^{\expr{f}} = u\),fix
\[LL_{\expr{f},\expr{y}} \colon \expr{y} \longrightarrow \expr{u}.\]
These choices are noncanonical, since at several steps in the construction of a light leaf one has to choose a rex path. In the computation we make all such choices deterministically, see Appendix~\ref{app:ll} for details. 

The associated double leaf is
\[\LL_{\expr{e},\expr{f}} := \overline{LL_{\expr{f}, \expr{y}}} \circ LL_{\expr{e},\expr{x}},\]
where the bar denotes a vertical flip.

The certificate defines polynomials
\[c_{\expr{f},\expr{e}}^{u} \in A_{0}[\alpha_{s},\alpha_{t},\alpha_{u}]\]
for pairs of subexpressions with common endpoint \(u\). We set
\begin{equation}
  \label{eq:cert}
  Z_{H_{3}} :=\sum_{u \in H_{3}} \sum_{\substack{\expr{e} \subset \expr{x}, \expr{f} \subset \expr{y} \\  \expr{x}^{\expr{e}} = u =  \expr{y}^{\expr{f}}}} c^{u}_{\expr{f},\expr{e}} \LL_{\expr{e},\expr{f}}.
\end{equation}
A line
\[\texttt{C e f a b c m n}\]
in the certificate contributes
\[(m+n\varphi)\alpha_{s}^{a}\alpha_{t}^{b}\alpha_{u}^{c}\]
to the polynomial \(c_{\expr{f},\expr{e}}^{u}\). The integers \texttt{e} and \texttt{f} encode subexpressions \(\expr{e}\) and \(\expr{f}\) respectively. Again, for details see Appendix~\ref{app:cert}.
Note that several lines of the certificate may contribute to one polynomial coefficient. After combining these contributions, there are \(43424\) nonzero polynomials \(c_{\expr{f},\expr{e}}^{u}\), represented by \(54452\) nonzero monomial records.

With this notation we can state our main computational result.

\begin{theorem}\label{th:rel}
  For the standard geometric realization of \(H_{3}\), we have
  \begin{equation}\label{eq:mainrel}\widetilde\Lambda(P_{L}) - \widetilde\Lambda(P_{R}) = \widetilde\Lambda(Z_{H_{3}}).\end{equation}
\end{theorem}

The proof occupies Sections~\ref{sec:comp} and~\ref{sec:implementation}. Assuming the theorem, we may define
\[\Hec_{BS}(H_{3}) := \widetilde{\Hec}_{BS}(H_{3}) / \langle P_{L} - P_{R} - Z_{H_{3}}\rangle_{\otimes}.\]
Theorem~\ref{th:rel} says exactly that \(\widetilde\Lambda\) factors through this quotient, we denote the induced functor by
\[\Lambda \colon \Hec_{BS}(H_{3}) \to \QG.\]
\begin{corollary}
  In \(\mathrm{BSBim}_{H_{3}}\) we have
  \[\widetilde{\CF}(P_{L} - P_{R} -Z_{H_{3}}) = 0.\]
  Consequently, \(\widetilde{\CF}\) factors through \(\Hec_{BS}(H_{3})\), we denote the induced functor by \(\CF\).
\end{corollary}
\begin{proof}
  By Theorem~\ref{th:rel} and the commutativity of diagram~\eqref{eq:comm-sq}, we have
  \[\widetilde{\CF}(P_{L} - P_{R} - Z_{H_{3}}) \otimes Q = 0.\]
  Bott-Samelson Hom spaces are free over \(R\) by the double leaves basis theorem \cite[Theorem 3.2]{Lib15}. Hence the natural map from such Hom spaces to their localizations is injective. Therefore
   \[\widetilde{\CF}(P_{L} - P_{R} -Z_{H_{3}}) = 0.\]
 \end{proof}
 Now we can draw a commutative diagram
\[
\begin{tikzcd}[column sep=large, row sep=large]
  \widetilde{\Hec}_{BS}(H_{3})
    \arrow[dr, two heads]
    \arrow[ddr, bend right, "\widetilde{\Lambda}"]
    \arrow[drr, bend left, "\widetilde{\mathcal{F}}"]
  &
  \\
  &
  \Hec_{BS}(H_{3})
    \arrow[r, "\mathcal{F}"]
    \arrow[d, "\Lambda"]
  &
 \BSB
    \arrow[d, "{(-)\otimes_R Q}"]
  \\
  &
 \QG
    \arrow[r, "\sim"]
  &
  \mathrm{StdBim}_{Q,H_3}
\end{tikzcd}
\]
which will be used later. 
\section{Computation}\label{sec:comp}

Let us take a closer look at the relation we want to compute.  After applying \(\widetilde\Lambda\), an expression \(\expr{w} = (w_{1}, w_{2}, \dots, w_s)\) splits as
\[
\widetilde{\Lambda}(\expr{w})
=
\bigoplus_{\expr{e}\subset\expr{w}} r_{\expr{w}^{\expr{e}}}.
\]

As a left \(Q\)-module, each such summand has rank one, but its right action is
twisted by the endpoint \(\expr{w}^{\expr{e}}\). For a subexpression \(\expr{e} \subset \expr{w}\), write
\[r_{\expr{e}}:= r_{\expr{w}^{\expr{e}}}.\]
This is an abuse of notation: the same subexpression for different words gives different endpoints, the word, though, will always be clear from the context.

Denote the localization of the difference by
\[L := \widetilde\Lambda (P_{L}) - \widetilde\Lambda (P_{R}).\]
Then
\[L\colon  \bigoplus_{\expr{e} \subset \expr{x}} r_{\expr{e}} \longrightarrow \bigoplus_{\expr{f} \subset \expr{y}} r_{\expr{f}} \]
For the canonical projectors and inclusions we have
\[p_{\expr{f}} \circ  L \circ \imath_{\expr{e}} \in \Hom_{\QG}\left(r_{\expr{e}}, r_{\expr{f}}\right)\cong \begin{cases}
Q,
&
\text{if }
\expr{x}^{\expr{e}}
=
\expr{y}^{\expr{f}},
\\
0,
&
\text{otherwise.}
\end{cases}\]
Thus, \(L\) may be represented by a sparse  \(2^{15}\times 2^{15}\) matrix with entries in \(Q\).

The aim is to express \(L\) in the basis formed by the localized double leaves
\(\widetilde{\Lambda}(\LL_{\expr{e},\expr{f}})\). Naively, we have an algorithm for computing double leaves, as well as an explicit realization of them under the functor \(\widetilde \Lambda\), so hypothetically one could construct and invert the  complete change-of-basis matrix. But, all of this is happening over the field \(K(\alpha_{s}, \alpha_{t},\alpha_{u})\), which is a computational nightmare, so we will not proceed like this. Instead, we should dive deep into the structure of the morphism \(L\) and double leaves.

\subsection{Decomposition by endpoints}
For \(u \in H_{3}\) , let
\[L^u\colon \bigoplus_{\substack{  \expr{e}\subset\expr{x}\\  \expr{x}^{\expr{e}}=u}}r_{\expr{e}}\longrightarrow\bigoplus_{\substack{  \expr{f}\subset\expr{y}\\  \expr{y}^{\expr{f}}=u}}r_{\expr{f}}\]
be the corresponding block of \(L\). Then
\[
L=\bigoplus_{u\in H_3}L^u.
\]

The degree constraint gives \(94\) possible endpoints in the coefficient expansion. Unfortunately, every double leaf can contribute to several different blocks, so we cannot simply proceed block by block. Instead we use the triangularity of double leaves with respect to the path-dominance order.

\subsection{Path-dominance triangularity}
Let \(\expr{w} = (s_{1}, \dots, s_{m})\) be an expression. For a subexpression \(\expr{e}\subset \expr{w}\) define its partial products by
\[
w_k^{\expr{e}}
:=
s_1^{e_1}\cdots s_k^{e_k},
\]
For two subexpressions \(\expr{e}, \expr{e}'\) of the same expression \(\expr{w}\), write

\[
\expr{e}\preceq \expr{e}' \qquad  \text{if} \   w_{k}^{\expr{e}}\leq w_{k}^{\expr{e}'} , \ \text{in the Bruhat order for every }k=1,\dots,m.
\]
This is the path-dominance order.

Let \(\expr{x}\) and \(\expr{y}\) be arbitrary expressions  and let \(\expr{e}\subset\expr{x}\) and \(\expr{f}\subset\expr{y}\) have the same endpoint \(z = \expr{x}^{\expr{e}} = \expr{y}^{\expr{f}}\). After localization, the corresponding double leaf \(\LL_{\expr{e},\expr{f}}^{z}\colon \expr{x} \to \expr{y}\)  becomes
\[\widetilde \Lambda (\LL_{\expr{e}, \expr{f}}^{z}) \colon \bigoplus_{\expr{e'} \subset \expr{x}} r_{\expr{e}'} \longrightarrow\bigoplus_{\expr{f'} \subset \expr{y}} r_{\expr{f}'} \]
We use the following property of the path-dominance order.
\begin{lemma}[Upper-triangularity]\label{lem:triang}
  If the morphism \(p_{\expr{f}'} \circ \widetilde \Lambda (\LL_{\expr{e}, \expr{f}}^{z}) \circ \imath_{\expr{e}'}\) is nonzero, then
  \begin{itemize}
  \item the subexpressions \(\expr{e}'\) and \(\expr{f}'\) have the same endpoint \(u = \expr{x}^{\expr{e}'} = \expr{y}^{\expr{f}'}\)
  \item one has  \(\expr{e}' \preceq \expr{e}\) and \(\expr{f}' \preceq \expr{f}\) in the path-dominance order.
  \end{itemize}
  Moreover, the morphism  \(p_{\expr{f}} \circ \widetilde \Lambda (\LL_{\expr{e}, \expr{f}}^{z}) \circ \imath_{\expr{e}}\) is given by multiplication by a unit in \(R[\beta^{-1} \mid \beta \in \Phi]\).
  In particular, the common endpoint satisfies \(u \leq z\) in the Bruhat order.
\end{lemma}
\begin{proof}
  See \cite[Proposition 6.6, \S 6.3]{EW16} 
\end{proof}

\subsection{Matrix presentation}
Fix an endpoint \(u\). Order all subexpressions of \(\expr{x}\) and \(\expr{y}\) with endpoint \(u\) by any linear order extending the path-dominance order. We always write them from smaller to larger in this order. Thus, if \(\expr{e}_{1},\dots,\expr{e}_{m}\) are such subexpressions of \(\expr{x}\), then
\[\expr{e}_{i} \preceq \expr{e}_{j} \Rightarrow i \leq j,\]
and similarly for \(\expr{f}_{1}, \dots, \expr{f}_{n}\).

Fix some reduced expression \(\expr{u}\) of \(u\). For every \(\expr{e}_{i} \subset \expr{x}\) with endpoint \(u\), consider the corresponding light leaf
\[LL_{\expr{e}_{i}, \expr{x}} \colon \expr{x} \to \expr{u}.\]
The localized morphism  \(LL_{\expr{e}_{i},\expr{x}}\), restricted to the endpoint \(u\) has target only in \(r_{u}\), since \(\expr{u}\) is reduced:
\[\left.\widetilde\Lambda (LL_{\expr{e}_{i},\expr{x}})\right|_{u} \colon \bigoplus_{j =1}^{m}r_{\expr{e}_{j}} \longrightarrow r_{u}\]
For each \(\expr{e}_{i}\), this morphism can be seen as a row \((A^{u}_{\expr{e}_{i},\expr{e}_{1}}, \dots, A^{u}_{\expr{e}_{i},\expr{e}_{m}})\). Varying index \(i\) we obtain a matrix which we denote by \(A^{u}\). 

For every \(\expr{f}_{j} \subset \expr{y}\) with endpoint \(u\), we have the flipped light leaf
\[\overline{LL_{\expr{f}_{j},\expr{y}}} \colon \expr{u} \longrightarrow \expr{y}\]
which, in turn, after localization and restriction to the endpoint \(u\) gives the morphism
\[ \left.\widetilde\Lambda \left(\overline{LL_{\expr{f}_{j},\expr{y}}}\right)\right|_{u} \colon r_{u} \longrightarrow \bigoplus_{i = 1}^{n} r_{\expr{f}_{i}}\]
represented as a column \[\begin{pmatrix} T^{u}_{\expr{f}_{1},\expr{f}_{j}} \\ \vdots \\ T^{u}_{\expr{f}_{n},\expr{f}_{j}} \end{pmatrix}.\]

Varying the index \(j\) we obtain the matrix \(T^{u}\).

By Lemma~\ref{lem:triang} \(A^{u}\) is lower triangular and \(T^{u}\) is upper triangular, both with diagonal entries in \(R[\beta^{-1} \mid \beta \in \Phi]^{\times}\).  Recall that a double leaf \(\LL_{\expr{e}, \expr{f}}\) is the composition
\[\expr{x} \xrightarrow{LL_{\expr{e}, \expr{x}}} \expr{u} \xrightarrow{\overline{LL_{\expr{f}, \expr{y}}}} \expr{y}.\]
Therefore, the \(u\)-block of the localized double leaf is a rank-one matrix
\begin{equation}
  \label{eq:matrix_presentation}
  \left.\widetilde\Lambda \left(\LL_{\expr{e}, \expr{f}}\right)\right|_{u} = \begin{pmatrix} T^{u}_{\expr{f}_{1},\expr{f}} \\ \vdots \\ T^{u}_{\expr{f}_{n},\expr{f}} \end{pmatrix} (A^{u}_{\expr{e},\expr{e}_{1}}, \dots, A^{u}_{\expr{e},\expr{e}_{m}}). 
\end{equation} 
Now, if we define the matrix \(C^{u} = (c_{\expr{f},\expr{e}}^{u})_{\expr{f},\expr{e}}\), then the contribution of double leaves with intermediate endpoint \(u\) to the \(u\)-block is 
\begin{equation}\label{eq:tca}\sum_{\substack{
  \expr{e}\subset\expr{x}\\
  \expr{f}\subset\expr{y}\\
  \expr{x}^{\expr{e}}
  =
  \expr{y}^{\expr{f}}
  =
  u
}}
c_{\expr{f},\expr{e}}^{u}
\left.
\widetilde\Lambda
\left(
\LL_{\expr{e},\expr{f}}
\right)
\right|_{u}
=
T^{u}C^{u}A^{u}.
\end{equation}
\subsection{Descending induction on endpoint length}\label{ssec:induction}

The equation~\eqref{eq:mainrel} is equivalent to

\begin{equation}
  \label{eq:main-rel2}
  L = \sum_{z\in H_{3}} \sum_{\substack{\expr{e} \subset \expr{x}, \expr{f} \subset \expr{y} \\  \expr{x}^{\expr{e}} = z =  \expr{y}^{\expr{f}}}} c^{z}_{\expr{f},\expr{e}} \widetilde\Lambda(\LL_{\expr{e},\expr{f}}).
\end{equation}
We compute the coefficients by descending induction on the length of the common endpoint.

Suppose that all coefficients with endpoint \(z\) satisfying \(\ell(z) > k\) are known. Fix \(u\in H_{3}\) with \(\ell(u) = k\). By Lemma \ref{lem:triang}, only intermediate endpoints \(z \geq u\) can contribute to the endpoint block \(u\). Hence the known contribution from higher endpoints is
\[H^{u} = \sum_{z > u}\sum_{\substack{\expr{e} \subset \expr{x}, \expr{f} \subset \expr{y}, \\ \expr{x}^{\expr{e}} = z = \expr{y}^{\expr{f}}}} c_{\expr{f},\expr{e}}^{z} \left.\widetilde\Lambda(\LL_{\expr{e},\expr{f}})\right|_{u}.\]
Put

\[R^{u} = L^{u} - H^{u}.\]

After this subtraction, only double leaves with endpoint \(u\) remain. Therefore, by~\eqref{eq:tca} we obtain
\begin{equation}\label{eq:matrices}R^{u} = T^{u} C^{u} A^{u}.\end{equation}

Since \(A^{u}\) and \(T^{u}\) are triangular with nonzero diagonal, \(C^{u}\) is obtained by two triangular substitutions. Thus, we do not actually have to invert any matrices directly.

The induction starts at \(w_{0}\). Both rex morphisms have coefficient \(1\) on the all-ones summand, so the \(w_{0}\) block of \(L\) is zero. In fact, the first nonzero coefficients occur at length \(11\), we explain it later in Remark~\ref{rem:1234}. 

\begin{remark}
  The computations of endpoints \(u\) of the same length are independent of each other, since these endpoints are not comparable in the Bruhat order. Therefore the above computation can be performed in parallel for every endpoint, which we heavily exploit. 
\end{remark}

\subsection{Exact interpolation}\label{ssec:intr}
With all this in hand, one can already proceed with the computations. Indeed, matrix equation~\eqref{eq:matrices} is more manageable than one might expect from the naive approach. However, it still requires billions of symbolic operations with rational functions. Packages such as GAP or Sage can in principle perform these operations, but for our purposes, this is too slow. Therefore, we instead evaluate the equations at sufficiently many points and then recover the coefficients by interpolation. 

For a subexpression \(\expr{e}\subset(s_1,\dots,s_m)\), define
\begin{align*}
\dft(\expr{e})
={}&
\#\left\{i:e_i=0,
\ \ell(w_{i-1}^{\expr{e}}s_i)>\ell(w_{i-1}^{\expr{e}})\right\}\\
&-
\#\left\{i:e_i=0,
\ \ell(w_{i-1}^{\expr{e}}s_i)<\ell(w_{i-1}^{\expr{e}})\right\}.
\end{align*}
This is the defect of \(\expr{e}\). The degree of the light leaf \(LL_{\expr{e},\expr{x}}\) is equal to \(\dft(\expr{e})\). Therefore the degree of a double leaf is 
\[\operatorname{deg}\LL_{\expr{e}, \expr{f}} =\dft(\expr{e}) + \dft(\expr{f}).\]
Since \(L\) has degree zero and every simple root has degree \(2\), a nonzero coefficient \(c_{\expr{f},\expr{e}}^{u}\) must be homogeneous of ordinary polynomial degree
\begin{equation}\label{eq:deg-bound}d_{\expr{e},\expr{f}}:= - \frac{\dft(\expr{e}) + \dft(\expr{f})}{2},\end{equation}
equivalently of graded degree \(2d_{\expr{e},\expr{f}}\).

So, we can rewrite~\eqref{eq:mainrel} for the last time as follows:
\begin{equation}
  \label{eq:main-ref2}
  L = \sum_{u \in H_{3}} \sum_{\substack{\expr{e} \subset \expr{x}, \expr{f} \subset \expr{y} \\  \expr{x}^{\expr{e}} = u =  \expr{y}^{\expr{f}}}} \sum_{\substack{a+b+c = d_{\expr{e},\expr{f}} \\a,b,c \geq 0}} c^{a,b,c}_{\expr{f},\expr{e}} \alpha_{s}^{a}\alpha_{t}^{b}\alpha_{u}^{c} \widetilde\Lambda(\LL_{\expr{e},\expr{f}})
\end{equation}
Now the coefficients \(c_{\expr{f},\expr{e}}^{a,b,c}\) belong to \(\mathbb{Q}(\varphi)\). The number of possible monomials of ordinary degree \(d\) is equal to \[N_{d} = \binom{d+2}{d}.\]
Consequently, it is enough to evaluate the equation~\eqref{eq:matrices} at \(N_{d}\) suitably chosen points. For a point \(p = (p_{s},p_{t},p_{u})\) we specialize  \[\alpha_{s} = p_{s}, \quad \alpha_{t} = p_{t}, \quad \alpha_{u} = p_{u}.\]
After subtracting the terms already determined at larger endpoint lengths, the program computes \(C^{u}[p]\), as before, from the equation
\[R^{u}[p] = T^{u}[p] C^{u}[p] A^{u}[p]\]
with all entries in \(\mathbb{Q}(\varphi)\). To recover the polynomial coefficient one solves a (relatively) small linear system
\begin{equation}
\label{eq:interpolation-system}
\begin{pmatrix}
 p_{1,s}^{a_1}p_{1,t}^{b_1}p_{1,u}^{c_1} & \cdots & p_{1,s}^{a_{N_d}}p_{1,t}^{b_{N_d}}p_{1,u}^{c_{N_d}} \\
 \vdots & \ddots & \vdots \\
 p_{N_d,s}^{a_1}p_{N_d,t}^{b_1}p_{N_d,u}^{c_1} & \cdots & p_{N_d,s}^{a_{N_d}}p_{N_d,t}^{b_{N_d}}p_{N_d,u}^{c_{N_d}}
\end{pmatrix}
\begin{pmatrix}
 c^{a_1,b_1,c_1}_{\expr{f},\expr{e}} \\
 \vdots \\
 c^{a_{N_d},b_{N_d},c_{N_d}}_{\expr{f},\expr{e}}
\end{pmatrix}
=
\begin{pmatrix}
 c^{u}_{\expr{f},\expr{e}}(p_1) \\
 \vdots \\
 c^{u}_{\expr{f},\expr{e}}(p_{N_d})
\end{pmatrix}.
\end{equation}
The only things to take care of are avoiding the root hyperplanes and choosing the points so that the matrix in~\eqref{eq:interpolation-system} is nonsingular.

\begin{remark}
  This reduction may seem questionable, since we specialize both the coefficients and the matrices simultaneously. We explain in the next section why this procedure is valid.
\end{remark}
\begin{remark}\label{rem:1234} 
For endpoint lengths \(12,13,14\) every double leaf with matching source and target endpoint has positive degree, therefore none of these double leaves can occur in a degree zero relation. Indeed, the first nonzero coefficient occurs only in the length \(11\). 
\end{remark}

In practice, we use one common collection of exact integer points for each length layer. The number of points required by a layer is determined by the largest value of \(d\) which occurs among its double leaves; coefficients of smaller degree simply use a smaller subset of the same collection. The numbers of determining points for the non-empty layers of lengths \(1,\ldots,11\) are
\[
1,\ 3,\ 6,\ 10,\ 15,\ 10,\ 6,\ 6,\ 3,\ 3,\ 1,
\]
respectively. Thus the layer of length \(1\) requires only one point, while the largest requires \(15\). In particular, the maximal ordinary total degree appearing is \(4\). For every layer except the constant layer of length \(11\), the certificate also contains one additional verification point, which is not used for interpolation. This is unnecessary, but was added for our own peace of mind.

\section{Implementation and verification}\label{sec:implementation}

In the previous section we gave all necessary mathematical reductions for the computation to be possible. Now we briefly describe the implementation and verification of the result.

\subsection{The producer}
The producer is written in Rust. Even after all the reductions above, one still has to perform an enormous number of operations with constants from \(\mathbb{Q}(\varphi) = \mathbb{Q}(\sqrt{5})\). Therefore, it is important to make the arithmetic with these constants as cheap as possible both in time and in memory. We aggressively avoid unnecessary allocations and reuse memory whenever possible.

The implementation follows the descending induction of Subsection~\ref{ssec:induction}. For a fixed length all required interpolation points and all endpoints of this length are independent and are evaluated in parallel.

After specialization, all arithmetic is exact in \(\mathbb{Q}(\varphi)\). Localized columns and triangular matrices are stored sparsely, repeated light leaf, double leaf, and rex computations are cached.

A complete 32-thread run took about \(2\) hours and used about \(30\) GiB of RAM. The source code and exact command are provided in the repository.

\subsection{Verification}\label{sec:verification}

Unfortunately, the result of the computation is not manageable by hand, the final certificate contains \(54452\) nonzero monomial records. Therefore, it would not be honest, if we did not implement a more trusted verification procedure, independent of the Rust producer. The role of the verifier is played by a GAP program, which performs exact computations using the standard exact matrix operations provided by GAP.

The GAP program does not recompute the coefficients from scratch (otherwise, why not use it in the first place?). Instead, it checks that the coefficients produced by Rust give the required identity.

Namely, the verifier performs the following steps.

\begin{enumerate}
\item It constructs the geometric realization of \(H_{3}\) over \(\mathbb{Q}(\sqrt{5})\), and verifies the Coxeter relations. From this realization it reconstructs the length function and the Bruhat order.
\item It checks that the source and target expressions are indeed reduced expressions of the unique longest element \(w_{0}\).
\item It reconstructs all admissible coordinates \((\expr{e}, \expr{f},a,b,c)\) which can occur in the relation from the endpoints and defects. The verifier then checks that every line in the certificate is admissible.
\item\label{it:nonzero} It checks that no chosen point lies on an \(H_{3}\) root hyperplane.
\item\label{it:fullrank} For every layer \(l\) it reconstructs the maximal polynomial degree and checks that the chosen set of points \(P_{l}\) is sufficient for interpolation, namely it checks that the left-hand side matrix of~\eqref{eq:interpolation-system} has full rank.  
\item Finally, at every point \(p \in P_{l}\) as well as at additional verification point, the verifier evaluates both rex morphisms and all double leaves occurring in the certificate and checks the equality for every admissible block \(u\)
\begin{equation}\label{eq:verp}\left.L\right|_{u}[p] = \sum_{z,\expr{e}, \expr{f}} c_{\expr{f},\expr{e}}^{z}(p) \left.\widetilde\Lambda(\LL_{\expr{e},\expr{f}})\right|_{u}[p].\end{equation} 
\end{enumerate}
There are \(22\) specialization points which have to be checked. These checks are independent and can therefore be run in parallel, which can be controlled manually, see README.md in the github repository.

\begin{remark}\label{rem:hard}
This computation is very CPU-consuming. The published certificate passed all \(22\) exact verification jobs. A 22-thread run (this is maximal parallelism here) required about \(15\) hours and about \(50\) GiB of RAM.
\end{remark}

\subsection{Proof of the Theorem~\ref{th:rel}}
  For the geometric realization consider two Bott-Samelson bimodules corresponding to expressions \(\expr{x} \) and \(\expr{y}\)
  \[B_{\expr{x}} = B_{u} \otimes_{R} B_{t} \otimes_{R} \dots \otimes_{R}B_{s}, \qquad B_{\expr{y}} = B_{s}\otimes_{R} B_{t} \otimes_{R} \dots \otimes_{R}B_{u}.\]
  By \cite[Theorem 3.2]{Lib15}, the double leaves form an \(R\)-basis of \(\Hom_{\mathrm{BSBim}_{H_{3}}}(B_{\expr{x}}, B_{\expr{y}}).\)
  Hence there are unique polynomials \(q_{\expr{f},\expr{e}}^{z} \in R\), where \(z\) indicates the endpoint, such that
  \[\widetilde{\CF}(P_{L} - P_{R}) = \sum_{z,\expr{e},\expr{f}} q_{\expr{f},\expr{e}}^{z} \widetilde{\CF}(\LL_{\expr{e},\expr{f}}).\]
  Applying functor \(- \otimes Q\) by the commutativity of the diagram~\eqref{eq:comm-sq} we have
  \begin{equation}\label{eq:pf2}L = \sum_{z,\expr{e},\expr{f}}q_{\expr{f},\expr{e}}^{z}\widetilde\Lambda(\LL_{\expr{e},\expr{f}}).\end{equation}
  Denote the difference \(\delta^{z}_{\expr{f},\expr{e}} = c_{\expr{f},\expr{e}}^{z} - q_{\expr{f},\expr{e}}^{z}\). We shall show that it is zero. Subtracting the specialization of \eqref{eq:pf2} from \eqref{eq:verp}, which is checked by verifier we obtain 
  \[\sum_{z
      \geq u} \sum_{\substack{\expr{e}\subset\expr{x},\expr{f}\subset\expr{y} \\ \expr{x}^{\expr{e}} = z = \expr{y}^{\expr{f}}}} \delta_{\expr{f},\expr{e}}^{z}(p) \left.\widetilde\Lambda(\LL_{\expr{e},\expr{f}})\right|_{u}[p] = 0.\]
  Suppose that there is at least one nonzero polynomial \(\delta_{\expr{f},\expr{e}}^{z}\), among those choose the one with the maximal \(l = \ell(z)\), therefore for every \(z' > z\) all \(\delta^{z'}_{\expr{f},\expr{e}}\) are zero polynomials. Therefore for the block \(z\) we have
\[\sum_{\substack{\expr{e}\subset\expr{x},\expr{f}\subset\expr{y} \\ \expr{x}^{\expr{e}} = z = \expr{y}^{\expr{f}}}} \delta_{\expr{f},\expr{e}}^{z}(p) \left.\widetilde\Lambda(\LL_{\expr{e},\expr{f}})\right|_{z}[p] = 0.\]
We argue that for any \(p \in P_{l}\), the set of matrices \[\left\{ \left.\widetilde\Lambda(\LL_{\expr{e},\expr{f}})\right|_{z}[p]\right\}_{\expr{x}^{\expr{e}} = z = \expr{y}^{\expr{f}}}   \subset \operatorname{Mat}_{n\times m}(K)\]  is \(K\)-linearly independent. Indeed, the equation~\eqref{eq:matrix_presentation} can be rewritten as
    \[\left.\widetilde\Lambda(\LL_{\expr{e},\expr{f}})[p]\right|_{z}= T^{z}[p] E_{\expr{f},\expr{e}} A^{z}[p],\]
    where \(E_{\expr{f},\expr{e}} \in \operatorname{Mat}_{n\times m}(K)\) is the standard matrix unit. The lemma~\ref{lem:triang} with the verifier's check at stage~\ref{it:nonzero} justifies that matrices \(T^{z}[p]\) and \(A^{z}[p]\) are triangular with nonzero diagonal, therefore they are invertible. Since \(E_{\expr{f},\expr{e}}\) are \(K\)-linearly independent, the same is true for \(\left.\widetilde\Lambda(\LL_{\expr{e},\expr{f}})\right|_{z}[p]\).
    Hence, for any \(p\) from \(P_{l}\) we have
    \[\delta_{\expr{f},\expr{e}}^{z}(p) = 0.\]
    The polynomial \(\delta^{z}_{\expr{f},\expr{e}}\) is homogeneous with ordinary degree \(d\), the set \(P_{l}\) is chosen so the kernel of specialization is zero (stage~\ref{it:fullrank}), therefore we conclude \(\delta_{\expr{f},\expr{e}}^{z} = 0\), which is contradiction. 

\subsection{Fast probabilistic verification}

In view of the remark~\ref{rem:hard}, we want to give anyone with any computer the opportunity to still, in some sense, check our result. For this purpose, in addition to the exact verification, the GAP program provides an optional fast probabilistic check based on Freivalds’ algorithm.
\begin{remark}
To be honest, this truly unnecessary piece of work was done not least to justify the two-semester computational complexity course which I was happy to attend as an undergraduate under Professors Daniil Musatov and Sergey Shestakov, who will probably never read this. 
\end{remark}

For every specialization point \(p\), and for every endpoint \(u\) we basically have to check a matrix equality of the form 
\[L^{u}[p] = N^{u}[p],\]
where
\[L^{u}[p] = \left.\widetilde\Lambda(P_{L})\right|_{u}[p] - \left.\widetilde \Lambda(P_{R})\right|_{u}[p]\]
and
\[N^{u}[p] = \sum_{z,\expr{e},\expr{f}} c_{\expr{f},\expr{e}}^{z}(p) \left.\widetilde\Lambda(\LL_{\expr{e},\expr{f}})\right|_{u}[p].\]
Instead of constructing and comparing the two matrices entry by entry, one can choose a random vector \(v\) and check only
\[L^{u}[p]v = N^{u}[p]v.\]
Both sides can be evaluated directly on \(v\), so constructing the full matrices is avoided. Suppose the coordinates of \(v\) are chosen independently and uniformly from
\[S = \{0,1,\dots,2^{32} -1 \} \subset \mathbb{Q}(\sqrt{5}).\]
If the claimed identity is correct, the test always passes. Suppose, on the other hand, that \(L^{u}[p] \neq N^{u}[p]\), and put
\[D = L^{u}[p] - N^{u}[p] \neq 0.\]
Choose a nonzero row of \(D\), and in this row choose a nonzero coefficient \(d_{j}.\) After all coordinates of \(v\) except \(v_{j}\) are fixed, the equation
\[Dv = 0\]
determines at most one possible value of  \(v_{j}\). Since \(v_{j}\) is chosen uniformly from \(S\), we obtain
\[\mathbb{P}(Dv = 0) \leq \frac{1}{|S|} = \frac{1}{2^{32}}.\]
Hence, if the certificate is incorrect, the probability that one complete run of the verifier still accepts it is at most \(2^{-32}\).

Repeating the test with independent random vectors decreases the probability of a false acceptance exponentially, after \(r\) independent runs it is at most \(2^{-32r}\).

\section{Categorification consequences}\label{sec:con}

We now explain how Theorem~\ref{th:rel} gives the relation for other realizations and how the standard consequences of \cite{EW16} and \cite{EW23} follow.

\subsection{Other \(H_{3}\) realizations}\label{ss:norm}
The Coxeter graph of the \(H_{3}\) is a tree, therefore  we can always balance the realization. Let \(\ch\) be an  \(H_{3}\) realization over a commutative domain \(\bk\), which
\[m_{st} =5, \quad m_{tu} = 3, \quad m_{su} = 2. \]
Put
\[x = -\langle\alpha_{s}^{\vee}, \alpha_{t}\rangle, \quad y = -\langle\alpha_{t}^{\vee}, \alpha_{s}\rangle, \quad z  = -\langle\alpha_{t}^{\vee}, \alpha_{u}\rangle, \quad w = -\langle\alpha_{u}^{\vee}, \alpha_{t}\rangle. \]

\begin{lemma}\label{lem:bal}
   In the notation above we have:
  \[(xy)^{2} - 3xy + 1 = 0, \quad zw = 1.\]
  If \(\rho = xy-1\), then
  \[\rho^{2} = \rho + 1,\]
  and \(\rho, x,y,z,w\) are invertible. Let
  \[\lambda_{s} = 1, \quad \lambda_{t} = \frac{\rho}{x}, \quad \lambda_{u} = \frac{\lambda_{t}}{z}.  \]
  After rescaling the simple roots and coroots by
  \[\beta_{r} = \lambda_{r}\alpha_{r}, \quad \beta^{\vee}_{r} =\lambda_{r}^{-1} \alpha_{r}^{\vee},\]
  the Cartan matrix of the realization becomes
  \[
    \begin{pmatrix}
      2 & -\rho & 0 \\
      - \rho & 2 & -1 \\
      0 & -1 & 2
    \end{pmatrix}.
  \]
\end{lemma}

\begin{proof}
  We use two-colored quantum numbers of \cite[\S 6.1]{EW23}. The relation\([5] = 0\) becomes
  \[(xy)^{2} - 3xy + 1 = 0\] and \([3] = 0\) becomes  \[zw - 1 = 0.\]
  Therefore \(\rho^{2} -\rho -1 = 0\) and all numbers included are invertible. The Cartan matrix computation is straightforward.
\end{proof}

\begin{proposition}\label{pro:rel}
  In the balanced normalization of Lemma~\ref{lem:bal}, the three-color \(H_{3}\) relation is obtained from the certificate by the base change
  \[A_{0} \longrightarrow \bk, \quad \varphi \mapsto \rho.\]
\end{proposition}

\begin{proof}
  Recall that the ring \(\mathbb{Z}[\varphi]\) is denoted by \(A_{0}\). Put
  \[R_{0} = A_{0}[\alpha_{s},\alpha_{t},\alpha_{u}],\]
  let \(\Theta \subset R_{0}\) be the multiplicative set generated by the set of roots \(\Phi\). All the localized matrices used in computation have entries in localized ring \(\Theta^{-1}R_{0}\). Since the certificate coefficients belong to \(R_{0}\), both sides of the identity in Theorem~\ref{th:rel} are matrices over \(\Theta^{-1}R_{0}\).
  Let
  \[R_{\ch} = \operatorname{Sym}_{\bk}(\ch^{*}), \qquad Q_{\ch} = \operatorname{Frac}(R_{\ch}).\]
  By Lemma~\ref{lem:bal}, the assignments 
  \[\psi\colon R_{0} \longrightarrow R_{\ch}, \qquad \varphi \longmapsto \rho, \qquad \alpha_{r} \longmapsto \beta_{r}\]
  define a \(H_{3}\)-equivariant ring homomorphism, moreover all \(w(\beta_{r})\) are nonzero. Now \(\psi\) extends to 
  \[\Theta^{-1}R_{0} \longrightarrow Q_{\ch}.\]
  Applying this homomorphism entry by entry to the matrix identity of Theorem~\ref{th:rel} gives the same relation, with the same choice of expressions, rex-paths and double leaves in the balanced normalization of the realization \(\ch\).  
\end{proof}

For the realization \(\ch\) we denote the relation obtained by the proposition above by \(Z_{H_{3}}^{\ch}\).

\subsection{Completion of the diagrammatic presentation} Let \((W,S)\) be a finite-rank Coxeter system and let \(\ch\) be a balanced realization over a commutative domain \(\bk\). Put
\[R = \operatorname{Sym}(\ch^{*}), \qquad  Q = \operatorname{Frac}(R).\]
Assume that for each finite dihedral parabolic subgroup the Jones-Wenzl projector exists and is rotatable. Define \(\Hec_{BS}(W)\) by the one-color, two-color relations of \cite{EW16}, for every finite rank-three parabolic subsystem of \((W,S)\) except the type \(H_{3}\) impose the same three-color relation as in \cite{EW16}. For every finite \(H_{3}\)-parabolic \(W_{J}\), impose three-color relation
\[P_{L} - P_{R} = Z_{H_{3}}^{\ch_{J}}.\]

Let \(\Hec(W)\) be the additive graded Karoubi envelope.

\begin{theorem}[Diagrammatic categorification]
  Under the assumption above:
  \begin{enumerate}
  \item there is a well-defined monoidal localization functor
    \[\Lambda_{W} \colon \Hec_{BS}(W) \longrightarrow \left(\Omega_{Q}W\right)_{\oplus},\]
    and the double leaves are linearly independent over \(R\);
  \item if \(\ch\) is Demazure surjective, the double leaves form an \(R\)-basis of every morphism space;
  \item if, in addition, \(\bk\) is a complete local ring, the usual character map gives an isomorphism
    \[K_{0}^{\oplus}(\Hec(W)) \cong \mathbf{H}_{W}\]
    of \(\mathbb{Z}[v,v^{-1}]\)-algebras.
  \end{enumerate}
\end{theorem}

\begin{proof}
  For \(H_{3}\), Proposition~\ref{pro:rel} gives a lower-term relation preserved by localization. As explained in \cite[\S 3.6]{EW23}, this supplies the missing
rank-three input. The remaining arguments of \cite[Theorem 1.6]{EW23} then apply.  Linear independence of double leaves follows as explained in \cite[\S 5.2]{EW23}. Assuming the Demazure surjectivity the spanning argument of \cite{EW16} applies again under the hypotheses of \cite[\S 5.2]{EW23}. If \(\bk\) is complete local, the final statement is \cite[Corollary 6.26]{EW16}
\end{proof}

\begin{corollary}[Soergel bimodules]
  Assume further that \(\bk\) is a field and that \(\ch\) is a Soergel realization satisfying the hypotheses of the comparison theorem of \cite[Theorem 6.30]{EW16}. Then the standard realization functor induces equivalences
  \[\Hec_{BS}(W) \xrightarrow{\sim} \operatorname{BSBim}, \qquad \Hec(W) \xrightarrow{\sim} \mathbb{S}\operatorname{Bim}.\]
\end{corollary}

\begin{remark}
  We have stated the results above for balanced realizations in order to avoid the additional notation. The corresponding statements for unbalanced realizations can be obtained by making the modifications of \cite[\S7]{EW23}.
\end{remark}

\appendix
\section{Certificate}\label{app:cert}

The certificate is a plain-text encoding of \(Z_{H_{3}}\) together with the auxiliary data, which is used by the verifier. Here we describe each of its blocks.

\subsection{Expressions and rex paths}
The records \texttt{SOURCE} and \texttt{TARGET} contain fixed expressions \(\expr{x}\) and \(\expr{y}\). The records \texttt{MOVE L} and \texttt{MOVE R} contain the two rex paths. For example the record

\[\texttt{MOVE L tu 6}\]
means that in the current expression of the left rex path, we apply the
\(\bluet\greenu\)-braid move starting at the sixth position:
\[
\greenu\bluet\reds\bluet\reds\,
\underline{\greenu\bluet\greenu}\,
\reds\bluet\greenu\reds\bluet\reds\bluet
\ \longrightarrow\
\greenu\bluet\reds\bluet\reds\,
\underline{\bluet\greenu\bluet}\,
\reds\bluet\greenu\reds\bluet\reds\bluet.
\]
Here \(m_{\bluet\greenu}=3\), so the indicated rex move replaces \(\greenu\bluet\greenu\) by \(\bluet\greenu\bluet\).
\begin{figure}[h!]
  \centering \includegraphics[width=\textwidth]{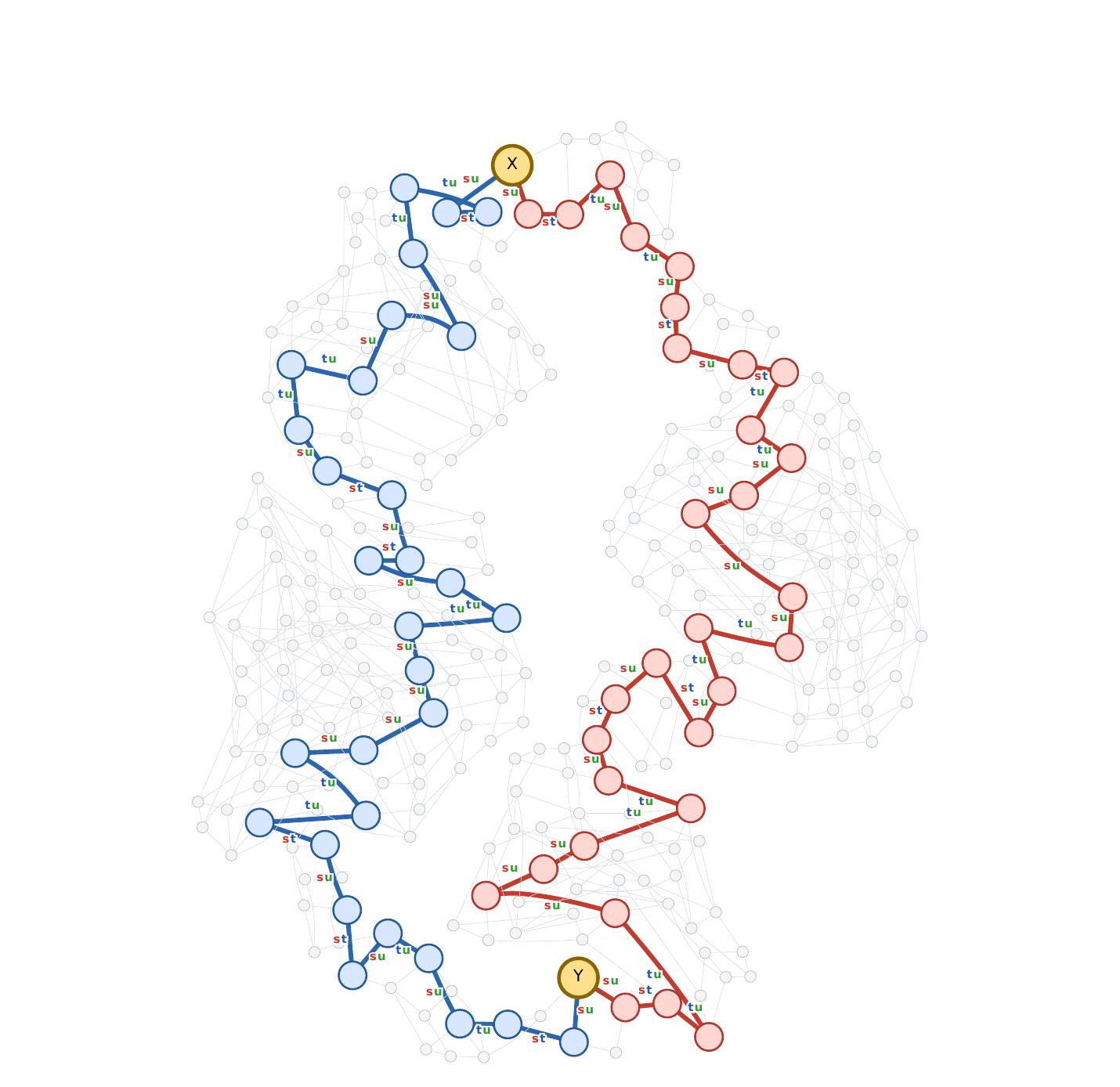}
  \caption{Two paths in the rex graph of \(w_0\in H_{3}\) connecting the expressions
    \(\expr{x}\) and \(\expr{y}\).}
  \label{fig:H3-rex-paths}
\end{figure}

\subsection{Points}
The next block is the list of points used for the interpolation. A record
\[\texttt{POINT k p\_s p\_t p\_u }\]
stores a point for the layer of endpoint length \(k\). For each layer except \(k = 11\), the final point is an additional verification point and is not used for interpolation. 

\subsection{Coefficients}
A coefficient record has the form
\[\texttt{C e f a b c m n}.\]
The integers \texttt{e} and \texttt{f} are bit masks for subexpressions of \(x\) and \(y\): the least significant bit controls the first leftmost letter of the word, and bit \(i\) controls the \(i+1\) letter. The record contributes
\[(m + n \varphi)\alpha_{s}^{a}\alpha_{t}^{b}\alpha_{u}^{c}\]
to the coefficient of \(\LL_{\expr{e},\expr{f}}\). Zero monomials and zero polynomial coefficients are omitted.

\begin{example}
  The block of records
\[
\begin{array}{l}
\texttt{C 12398 16435 0 0 1 1 2}\\
\texttt{C 12398 16435 0 1 0 -10 -17}\\
\texttt{C 12398 16435 1 0 0 -20 -29}
\end{array}
\]
encodes the following decorated double leaf.

\medskip
\begin{minipage}[c]{0.5\textwidth}

For the source expression
\[
\expr{x}=\greenu\bluet\reds\bluet\greenu\reds\bluet\reds\bluet\greenu\reds\bluet\reds\bluet\reds,
\]
the mask is
\[
  12398=(011000001101110)_2,
\]
\[
\expr{e}
=
(0,1,1,1,0,1,1,0,0,0,0,0,1,1,0).
\]
Here the binary expansion is written in the usual order, while
\(\expr{e}\) is written in expression order, with the least significant
bit corresponding to the first letter of \(\expr{x}\).
Thus the selected positions are
\[
2,3,4,6,7,13,14,
\]
giving
\[
\expr{x}^{\expr{e}}
=
\bluet\reds\bluet\reds\bluet\reds\bluet
=
\reds\bluet\reds.
\]

For the target expression
\[
\expr{y}=\reds\bluet\reds\bluet\reds\greenu\bluet\reds\bluet\reds\greenu\bluet\reds\bluet\greenu,
\]
we have
\[
16435=(100000000110011)_2,
\]
\[
\expr{f}
=
(1,1,0,0,1,1,0,0,0,0,0,0,0,0,1).
\]
Thus the selected positions are
\[
1,2,5,6,15,
\]
and
\[
\expr{y}^{\expr{f}}
=
\reds\bluet\reds\greenu\greenu= \reds\bluet\reds
\]
as well. 
\end{minipage}
\hfill
\begin{minipage}[c]{0.50\textwidth}
\centering
\exampleDecoratedDoubleLeafPicture
\end{minipage}
\end{example}

The certificate contains \(54\,452\) coefficient records.  After collecting
records with the same pair of masks, they give \(43\,424\) nonzero polynomial
coefficients.  Their distribution by intermediate endpoint length is shown in
Table~\ref{tab:certificate-layers}.

\begin{table}[H]
\centering
\begin{tabular}{c@{\qquad}r@{\qquad}r}
\toprule
\(\ell(u)\)
& nonzero polynomial coefficients
& monomial records\\
\midrule
11 & 24 & 24\\
10 & 213 & 231\\
9  & 397 & 409\\
8  & 2\,283 & 2\,636\\
7  & 4\,486 & 5\,193\\
6  & 7\,513 & 9\,761\\
5  & 15\,084 & 20\,961\\
4  & 8\,036 & 9\,476\\
3  & 3\,123 & 3\,404\\
2  & 2\,103 & 2\,195\\
1  & 162 & 162\\
\midrule
\textbf{Total} & \textbf{43\,424} & \textbf{54\,452}\\
\bottomrule
\end{tabular}
\caption{Nonzero coefficients in the certificate, grouped by the length of
the intermediate endpoint of the double leaf.}
\label{tab:certificate-layers}
\end{table}

\section{Light leaves construction}\label{app:ll}

In this section we describe the choices used to construct the double-leaf basis in the computation. This allows one to use the certificate whenever the same choices are made.

There are three of them to make: reduced expressions, expressions for the right descents, and rex paths. We fix the order \(\reds < \bluet < \greenu\) on generators and the corresponding lexicographic order on words.

\begin{enumerate}
\item \textit{Reduced expressions.}
For each \(w \in H_{3}\), we choose the first expression found by breadth-first search from the identity, using right multiplication by \(s,t,u\) in this order.
Since the search visits elements in increasing distance from the identity, the chosen expression is reduced.

\item \textit{Descent expressions.}
  For each \(w \in H_{3}\), and each right descent \(r\), we choose the lexicographically smallest reduced expression of \(w\) ending in \(r\).

\item \textit{Rex paths.}
  The procedure for the rex paths is the following. For a pair of expressions \(u,v\) where \(u\) is the source and \(v\) is the target we always choose the shortest path. We first compute the distance from every vertex \(u'\) to \(v\) by breadth-first search starting at the target. Then, starting at the source, we repeatedly take a braid move that decreases the distance by one. If there are several of them, we choose the smallest expression in lexicographic order.
\end{enumerate}
From this point on, the algorithm is standard and follows \cite{EW16} precisely.

\bibliographystyle{amsalpha}
\bibliography{references}

\end{document}